\documentclass[article]{amsart}
\usepackage{cases}
\usepackage{amsmath, amsfonts, pifont, amssymb, xcolor}
\usepackage{verbatim}
\usepackage[bookmarks=true]{hyperref}
\usepackage[margin=1 in]{geometry}
\usepackage[numbers,sort]{natbib}
\usepackage{mathtools}
\mathtoolsset{showonlyrefs}

\usepackage{tikz}
\usetikzlibrary{patterns}

\newtheorem{theorem}{Theorem}[section]
\newtheorem{lemma}[theorem]{Lemma}
\newtheorem{proposition}[theorem]{Proposition}

\newcommand{\R}{\mathbb{R}}
\newcommand{\bR}{\mathbb{R}}
\newcommand{\bZ}{\mathbb{Z}}
\newcommand{\bN}{\mathbb{N}}
\newcommand{\bT}{\mathbb{T}}

\newcommand{\beq}{\begin{equation}}
\newcommand{\eeq}{\end{equation}}
\newcommand{\beqq}{\begin{equation*}}
\newcommand{\eeqq}{\end{equation*}}

\newcommand{\w}{\widetilde}
\theoremstyle{definition}

\theoremstyle{remark}
\newtheorem{remark}[theorem]{Remark}

\numberwithin{equation}{section}

\def \l {\left}
\def \r {\right}

\def \lea {\lesssim}
\def \ca { \mathbf{1}}

\numberwithin{equation}{section}

\allowdisplaybreaks[4]

\begin{document}

\address{Yangkendi Deng
\newline \indent Department of Mathematics and Statistics, Beijing Institute of Technology.
\newline \indent Key Laboratory of Algebraic Lie Theory and Analysis of Ministry of Education.
\newline \indent  Beijing, China. \indent}
\email{dengyangkendi@bit.edu.cn}

\address{Chenjie Fan
\newline \indent  State Key Laboratory of Mathematical Sciences, Academy of Mathematics and Systems Science, Chinese Academy of Sciences
\newline \indent Beijing, China.
\indent }
\email{fancj@amss.ac.cn}

\address{Zehua Zhao
\newline \indent Department of Mathematics and Statistics, Beijing Institute of Technology.
\newline \indent Key Laboratory of Algebraic Lie Theory and Analysis of Ministry of Education.
\newline \indent  Beijing, China. \indent}
\email{zzh@bit.edu.cn}

\title[Sharp Strichartz estimate for Hyperbolic Schr\"odinger equation on $\mathbb{R}\times \mathbb{T}$]{Sharp $L^4$ Strichartz estimate for Hyperbolic Schr\"odinger equation on $\mathbb{R}\times \mathbb{T}$}
\author{Yangkendi Deng, Chenjie Fan and Zehua Zhao}


\maketitle

\begin{abstract}

We prove the \textit{sharp} $L^4$ Strichartz estimate without derivative loss for the hyperbolic Schr\"odinger equation on $\mathbb{R}\times\mathbb{T}$,
\begin{equation}
\|e^{it (\partial_{x_1}^2-\partial_{x_2}^2)} \phi\|_{L^4_{t,x_1,x_2}([0,1]\times \bR\times \bT)}\lea \|\phi\|_{L_{x_1,x_2}^2(\bR\times \bT)},   
\end{equation}
  which serves as the hyperbolic analogue of the classical result of Takaoka–Tzvetkov \cite{takaoka20012d}. The proof is based on the combination of a robust kernel decomposition method with precise measure estimates for semi-algebraic sets. As an immediate application, we establish the global well-posedness for the cubic hyperbolic Schr\"odinger equation on $\mathbb{R}\times\mathbb{T}$ in the $L^2$-critical space with sufficiently small initial data.

\end{abstract}

\bigskip

\setcounter{tocdepth}{1}
\tableofcontents

\parindent = 10pt     
\parskip = 8pt

\section{Introduction}

\subsection{Background and Motivations}

In this article, we consider  Strichartz estimates for linear hyperbolic Schr\"odinger equation:
\begin{equation}\label{eq: linear eqution}
\begin{cases}
i \partial_t u+\square u=0,\\
u(0,x)=\phi(x)\in L^2(\bR\times \bT).	
\end{cases}
\end{equation}
where $\square$ denotes $\partial_{x_1}^2-\partial_{x_2}^2$ and $x=(x_1,x_2)$. We denote the solution to linear equation \eqref{eq: linear eqution} by $e^{it\square}\phi$ or $H(t)\phi$.

The theory of $L^4$ Strichartz estimates for  (elliptic) Schr\"odinger equation is now relatively complete on the two-dimensional spaces $\mathbb{R}^2$, $\mathbb{R}^2 \times \mathbb{T}$, and $\mathbb{T}^2$. Without distinguishing the underlying spaces, we denote the solution to the linear Schr\"odinger equation with initial data $\phi \in L^2$ by $e^{it\Delta}\phi$,
\begin{equation}
\begin{cases}
i\partial_t u+\Delta u=0,\\
u(0,x)=\phi(x).
\end{cases}
\end{equation}

Strichartz estimates on $\mathbb{R}^2$,
$$  \|e^{it\Delta} \phi\|_{L^4_{t,x}(\bR \times \mathbb{R}^2)} \lesssim \|\phi\|_{L_x^2(\mathbb{R}^2)}  $$
 follow from classical $TT^*$ method and the dispersive estimate, see \cite{cazenave2003semilinear,keel1998endpoint}. Strichartz estimates on $\mathbb{T}^2$, 
 $$  \|e^{it\Delta} P_{\le N} \phi\|_{L^4_{t,x}([0,1] \times \mathbb{T}^2)} \lesssim_\varepsilon N^\varepsilon \|\phi\|_{L_x^2(\mathbb{T}^2)}, \forall \varepsilon>0  $$
were proved in \cite{Bourgain1} via  techniques from analytic number theory, and also follow from  $\ell^2-$decoupling theory, \cite{BD}. In a notable recent development, \cite{Herr},  the following sharp estimate 
$$  \|e^{it\Delta} P_{\le N} \phi\|_{L^4_{t,x}([0,1] \times \mathbb{T}^2)} \lesssim (\log N)^{\frac14} \|\phi\|_{L_x^2(\mathbb{T}^2)}  $$
was obtained by employing methods from incidence geometry. On $\mathbb{R} \times \mathbb{T}$, Takaoka and Tzvetkov established the estimate 
$$  \|e^{it\Delta} \phi\|_{L^4_{t,x}([0,1] \times \mathbb{R} \times \mathbb{T})} \lesssim \|\phi\|_{L_x^2(\mathbb{R}\times \mathbb{T})}  $$
by analyzing the measure of annuli on $\mathbb{R} \times \mathbb{Z}$, \cite{takaoka20012d}. It is also noteworthy that there exists a class of global-in-time Strichartz estimates on $\mathbb{R} \times \mathbb{T}$ or on higher-dimensional waveguides, see \cite{Barron} and \cite{BCP21}.

Let us turn to two-dimensional hyperbolic Schr\"odinger equation.  Same $L^{4} $ Strichartz estimates as the elliptic case on $\mathbb{R}^2$ holds by again applying  $TT^*$ method and the dispersive estimate. Thus, we focus on the cases of $\mathbb{T}^2$ and $\mathbb{R} \times \mathbb{T}$. Wang \cite{wang2013periodic} established the following estimate for $\mathbb{T}^2$: 
$$ \|e^{it\square} P_{\le N} \phi\|_{L^4_{t,x}([0,1] \times \mathbb{T}^2)} \lesssim N^{\frac14} \|\phi\|_{L_x^2(\mathbb{T}^2)}.$$ 
It is worth noting that, in $\bT^{2}$ case, according to recent work \cite{liu2025}, or more precisely as shown in \cite{BD3}, $L^{6}$ estimates also play a significant role. On the space $\bR\times \bT$ for the hyperbolic Schr\"odinger equation, we need to distinguish between the linear evolution operators $e^{it\square}$ and $e^{it\partial_{x_1} \partial_{x_2}}$. Recent works \cite{bacsakouglu2024local,BSTW25} unify results in both directions, establishing estimates 
\begin{equation}\label{eq:hyperbolic with loss}
\|e^{it\square} P_{\le N} \phi\|_{L^4_{t,x}([0,1] \times \bR\times \mathbb{T})} \lesssim (\log N)^{\frac14} \|\phi\|_{L_x^2(\bR\times \bT)},
\end{equation}
and 
$$\|e^{it\partial_{x_1} \partial_{x_2}} P_{\le N} \phi\|_{L^4_{t,x}([0,1] \times \bR\times \mathbb{T})} \lesssim N^{\frac14} \|\phi\|_{L_x^2(\bR\times \mathbb{T})}.$$
See also the recent work \cite{demeter2025restriction} for restriction and decoupling estimates for the hyperbolic paraboloid.

The main result of this paper (Theorem \ref{thm:hyperbolic on RT}) removes the $(\log N)^{\frac14}$ loss in estimate \eqref{eq:hyperbolic with loss}. Based on this critical $L^4$ Strichartz estimate, we naturally obtain local well-posedness for the cubic hyperbolic nonlinear Schr\"odinger equation

\begin{equation}\label{eq: main}
\begin{cases}
i \partial_t u+\square u=\pm |u|^2 u,\\
u(0,x)=\phi(x)\in L_x^2(\bR\times \bT),
\end{cases}
\end{equation}
with arbitrary $L^2$ initial data, as well as global well-posedness for small ${L}^2$ initial data. See Theorem \ref{mainthm: 2}.

The hyperbolic nonlinear Schr\"odinger equation (HNLS) serves as a fundamental model in multiple physical contexts, such as nonlinear optics with anisotropic dispersion and the theory of water waves. See the recent survey \cite{saut2024hyperbolic} and the references therein. It also emerges in kinetic theory, with HNLS-type dynamics being observed in high-frequency limits of the Boltzmann equation. A closely related and physically important extension is the (hyperbolic-elliptic) Davey-Stewartson (DS) system originated in \cite{davey1974three}, a nonlocal model which reduces to HNLS in certain regimes (see \cite{rosenzweig2018global} and the references therein). Therefore, a thorough understanding of HNLS provides a crucial stepping stone for analyzing the more complex dynamics of the DS system. We also refer to recent work \cite{Herrnew} for KP-II equations on the cylinder, which also involves \eqref{eq: linear eqution}. It would also be of great interest to pursue large-data global well-posedness for cubic hyperbolic nonlinear Schr\"odinger equation
\eqref{eq: main}.

\subsection{Statement of main results}
We are now ready to present the main results in this paper.
First, we show the {sharp} $L^4$-Strichartz estimate for hyperbolic Schr\"odinger equation on the cylindrical domain as follows,
\begin{theorem}[Sharp $L^4$-Strichartz estimate]\label{thm:hyperbolic on RT}
The estimate
$$ \|e^{it \square} \phi\|_{L^4_{t,x}([0,1]\times \bR\times \bT)}\lea \|\phi\|_{L_x^2(\bR\times \bT)}$$
holds for $\phi\in L_x^2(\bR\times \bT)$.
\end{theorem}

\begin{remark}
The above estimate is local-in-time. It is expected to obtain a global-in-time Strichartz estimate motivated by \cite{BCP21} as follows, 
\begin{equation}
 \|e^{it \square} \phi\|_{\ell^8_\gamma(\bZ) L^4_{t,x}([\gamma,\gamma+1]\times \bR\times \bT)}\lea \|\phi\|_{L_x^2(\bR\times \bT)}.
\end{equation}
See \cite{BCP21} for the global spacetime norm $\ell^8_\gamma(\bZ) L^4_{t,x}$. Such global-in-time refinements are expected and would be instrumental for small-data scattering in other critical models, such as the 2D quintic NLS; see \cite{BSTW25}.

\end{remark}

Next, we have well-posedness theory for \eqref{eq: main},
\begin{theorem}[Well-posedness theory]\label{mainthm: 2} The Cauchy problem
\eqref{eq: main} is local well-posed for arbitrary initial data $\phi\in L^2(\bR\times\bT)$ and globally well-posed for sufficiently small data $\phi\in L^2(\bR\times\bT)$.
\end{theorem}

\begin{remark}
Theorem \ref{mainthm: 2} follows from the sharp Strichartz estimate in Theorem \ref{thm:hyperbolic on RT} in a standard way. See \cite{takaoka20012d} for the elliptic case.   
\end{remark}
Moreover, we establish sharp bilinear Strichartz estimates in the hyperbolic setting. Sharp bilinear estimates of this type are crucial in applications to frequency-localized nonlinear interactions and are of independent interest. 
\begin{theorem}[Bilinear Strichartz estimates]\label{mainthm: 3}
Let $N_1\gg N_2 \ge 1, \lambda\ge 1$, $N_1, N_2 \in \bN$ and $\bT_\lambda=\bR/(\lambda\bZ)$. Assume that $\phi_1, \phi_2\in L^2(\bR\times \bT_\lambda)$. Then there holds
$$ \|e^{it \square} P_{N_1} \phi_1 \cdot e^{it \square} P_{N_2} \phi_2 \|_{L^2_{t,x}([0,1]\times \bR\times \bT_\lambda)}\lea \l(\frac{1}{\lambda}+\frac{N_2}{N_1}\r)^{\frac12} \|\phi_1\|_{L_x^2}  \|\phi_2\|_{L_x^2}.$$
\end{theorem}

\begin{remark}
This estimate recovers the elliptic bilinear eigenfunction bound in \cite{MR4782142}. Moreover, by comparing with the corresponding linear estimate, we see that the dependence on the frequency scales is optimal; thus the bilinear estimate is sharp.
\end{remark}

\subsection{Organization of the rest of this paper}
The remainder of this paper is structured as follows. Section \ref{3} is devoted to the proof of the sharp $L^4$ Strichartz estimate (Theorem \ref{thm:hyperbolic on RT}), which forms the core of our work; here, we detail the kernel decomposition method and the pivotal measure-theoretic arguments\footnote{The kernel decomposition method utilized here is robust and likely applicable to other multilinear estimates. See, for instance, the recent work \cite{DZZ}, where it is used in the study of the cubic NLS on the product space $\mathbb{R}\times\mathbb{S}^3$.}. The application of this linear estimate to establish the $L^2$-critical well-posedness theory for the cubic hyperbolic NLS (Theorem \ref{mainthm: 2}) is presented in Section \ref{4}, where we prove both local well-posedness and small-data global well-posedness. Following this, Section \ref{5} contains the proof of the sharp bilinear Strichartz estimate (Theorem \ref{mainthm: 3}), a result of independent interest that is crucial for analyzing nonlinear interactions in frequency-localized contexts. Finally, in Section \ref{6}, we synthesize our results, discuss their broader implications in the context of dispersive PDEs on waveguide manifolds, and outline a series of open problems for future research.

\subsection{Notations}
Throughout this note, we use $C$ to denote  the universal constant and $C$ may change line by line. We say $A\lesssim B$, if $A\leq CB$. We say $A\sim B$ if $A\lesssim B$ and $B\lesssim A$. We also use notation $C_{B}$ to denote a constant that depends on $B$. We use usual $L^{p}$ spaces and Sobolev spaces $H^{s}$. 
\subsection*{Acknowledgments}
The authors appreciate Prof. Nicolas Burq and Prof. Yuzhao Wang for helpful discussions. Y. Deng was supported by China Postdoctoral Science Foundation (Grant No. 2025M774191) and the NSF grant of China (No. 12501117). C. Fan was partially supported by the National Key R\&D Program of China, 2021YFA1000800, CAS Project for Young Scientists in Basic Research, Grant No.YSBR-031, and NSFC Grant (Nos. 12288201 \& 12471232). Z. Zhao was supported by the NSF grant of China (No. 12101046, 12426205, 12271032) and the Beijing Institute of Technology Research Fund Program for Young Scholars.

\section{Sharp Strichartz estimate: Proof of Theorem \ref{thm:hyperbolic on RT}}\label{3}

In this section, we give the proof of Theorem \ref{thm:hyperbolic on RT}. In our analysis, we often need to consider both the Euclidean measure on $\mathbb{R}^2$ and the product measure on $\mathbb{R} \times \mathbb{Z}$ for a set. To prevent confusion, we introduce the following notation to distinguish them.

For $\lambda\ge 1$ and a set $E\subset \bR^2$, we define $|E|_{\bR^2}$ is the measure of $E$ on $\bR^2$, and
$$ |E|_{\bR \times \bZ_{1/\lambda}}= \frac{1}{\lambda}\sum_{\xi_2\in \bZ_{1/\lambda}} \int_\bR  \ca_{E}(\xi_1, \xi_2) {\rm d}\xi_1. $$
We similarly define the measure $|\cdot|_{\bR}.$ Denote
$$ \int_{\bR\times \bZ_{1/\lambda}} F(\xi) {\rm d} \xi := \frac{1}{\lambda}\int_{\bR}\sum_{\xi_2\in \bZ_{1/\lambda}} F(\xi_1, \xi_2) {\rm d} \xi_1$$
and
$$ \|F\|_{L^2(\bR\times \bZ_{1/\lambda})}^2:= \frac{1}{\lambda}\int_{\bR}\sum_{\xi_2\in \bZ_{1/\lambda}} |F(\xi_1, \xi_2)|^2 {\rm d} \xi_1  $$
for simplicity.

Before the proof of Theorem \ref{thm:hyperbolic on RT}, we need a measure estimate lemma as follows. We adopt the standard definition and terminology for semi-algebraic sets as presented in \cite{Basu2024}. A set $E \subset \bR^{d}$ is called a \emph{semi-algebraic set} provided it can be represented as a finite union of subsets, each described by a system of polynomial equalities and strict inequalities of the form $P=0, Q_1>0,\cdots, Q_k>0$. If the total number of distinct polynomials used in the definition of $E$ is at most $s$, and the maximum degree among them does not exceed $D$, we refer to the product $sD$ as an upper bound for the \emph{complexity} of $E$.

It is worth noting that throughout this paper, all sets whose $\bR^2$ or $\bR\times \bZ$ measures we need to control will be semi-algebraic with complexity $\lea 1$.

\begin{lemma}[Lemma 4.2 in \cite{DWWZ2025}]\label{lem:measure control for RZ}
Let $\lambda \ge 1$ and let $E$ be a bounded semi-algebraic set on $\bR^{2}$ with complexity $\lea 1$. Then
$$  |E|_{\bR\times \bZ_{1/\lambda}} \lea |E|_{\bR^2}+     \frac{1}{\lambda}  \sup_{\xi_2 \in \bZ_{1/\lambda}}\int_{\bR } \mathbf{1}_{E}(\xi_1,\xi_2) {\rm d}\xi_1.      $$

\end{lemma}

\begin{remark}
The core idea of this lemma is that the function $\int_{\bR } \mathbf{1}_{E}(\xi_1,\xi_2) {\rm d}\xi_1 $ changes monotonicity with respect to $\xi_2$ only $O(1)$ times, from which the above estimate follows directly. When the set $E$ has a simple structure, this boundedness of the number of monotonicity changes can be verified by direct computation. For more complicated sets—in particular, for the semi-algebraic sets assumed in our lemma—this property is guaranteed by Lemma 2.9 in \cite{Basu2024}.
\end{remark}

\begin{remark} This lemma can be interpreted as follows: the total length of the dashed cross-sectional segments in the semi-algebraic set shown in Figure 1 is controlled by its area and the length of the longest cross-sectional segment.
Moreover, the boundedness condition as stated in Lemma \ref{lem:measure control for RZ} can be removed in a certain sense: proceed by considering the bounded semi-algebraic set $E_N = E \cap \{\xi \in \mathbb{R}^2 : |\xi| \le N\}$ and letting $N \to \infty$. It is important to note that the implicit constant in the aforementioned estimate depends solely on the dimension and the complexity of the semi-algebraic set.
\end{remark}

\begin{figure}[]\label{figure 1}
\centering
\begin{tikzpicture}[scale=1]
  \foreach \y in {0,0.333,0.666,1,1.333,1.666,2,2.333,2.666}
    \draw[dashed, gray] (0,\y) -- (4,\y);
  
  \fill[blue!20, opacity=0.7] 
    (1,1) .. controls (1.5,0.5) and (2,0.7) .. (2.5,1)
    .. controls (3,1.3) and (3,2) .. (2.5,2.2)
    .. controls (2,2.4) and (1.5,2.2) .. (1,2)
    .. controls (0.5,1.8) and (0.5,1.2) .. (1,1)
    -- cycle;
  \draw[blue, thick] 
    (1,1) .. controls (1.5,0.5) and (2,0.7) .. (2.5,1)
    .. controls (3,1.3) and (3,2) .. (2.5,2.2)
    .. controls (2,2.4) and (1.5,2.2) .. (1,2)
    .. controls (0.5,1.8) and (0.5,1.2) .. (1,1);
  
\end{tikzpicture}
\caption{A schematic illustration of the semi-algebraic set and its cross sections.}
\end{figure}
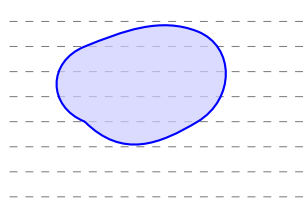

Throughout this section, we restrict to the case where $\lambda=1$.

Define $H(\xi)=\xi_1^2-\xi_2^2$, $H(\xi, \eta)=\xi_1\eta_1-\xi_2\eta_2$ for $\xi=(\xi_1,\xi_2)\in \bR^2, \eta=(\eta_1,\eta_2)\in \bR^2$. The solution to \ref{eq: linear eqution} is given by
$$ e^{it \square} \phi(x)=\int_{\bR\times \bZ} e(x\cdot \xi-t H(\xi))\widehat{\phi}(\xi) {\rm d}\xi,    $$
where $e(z)$ denotes $e^{2\pi i z}$ for $z\in \mathbb{C}$.

\begin{proof}[Proof of Theorem \ref{thm:hyperbolic on RT}]
Analytically, $\mathbb{R} \times \mathbb{Z}_{1/2}$ differs from $\mathbb{R} \times \mathbb{Z}$ only by a constant factor. Therefore, for the sake of simplicity in exposition, we will not distinguish between them through the proof. Let $\varphi(t)\in \mathcal{S}(\R)$ be such that: (1) $\widehat{\varphi}(\tau)\geq 0$ for all $\tau\in\R$; (2) the support of $\widehat{\varphi}$ lies in $[-\frac12,\frac12]$; (3)  $\varphi(t)\ge 0$ for $t\in \bR$, and $ \varphi(t)\ge 1$ for $t\in[-2,2]$.

A direct computation shows that 
\begin{align*}
    &\|e^{it \square} \phi\|_{L^4_{t,x}([0,1]\times \bR\times \bT)}^4\\
    \lea & \|e^{it \square} \phi \cdot (\varphi(t))^\frac14  \|_{L^4_{t,x}(\bR\times \bR\times \bT)}^4 \\
= & \int_{(\bR\times \bZ)^4} \delta_0( \xi-\eta+\gamma-h) \widehat{\varphi}\l(H(\xi)-H(\eta)+H(\gamma)-H(h)\r) \widehat{\phi}(\xi) \overline{\widehat{\phi}(\eta)} \widehat{\phi}(\gamma) \overline{ \widehat{\phi}(h)}   {\rm d} \xi {\rm d} \eta {\rm d} \gamma {\rm d} h.
\end{align*}
Define $f=|\widehat{\phi}|$ and make the linear change of variables 
\begin{align*}
    \left\{
    \begin{array}{lll}
        u&=& \frac12( \xi+\gamma)=\frac12(\eta+h) , \\
        v&=& \frac12(\xi-\gamma),\\
        w&=& \frac12( \eta-h). 
    \end{array}
    \right.
\end{align*}
Using $\widehat\varphi \lea \ca_{[-\frac12, \frac12]}$ and $$H(\xi)-H(\eta)+H(\gamma)-H(h)=2(H(v)-H(w)),$$ it suffices to prove that
\begin{equation}\label{eq:4-linear for hyperbolic RT}
\int_{\bR \times \bZ}\int_{\bR \times \bZ}\int_{\bR \times \bZ} \ca_{|H(v)-H(w)|\le 1} f(u+v)f(u-v)f(u+w)f(u-w) {\rm d}u {\rm d}v {\rm d}w \lea \|f\|_{L^2}^4,
\end{equation}
where $f\in L^2(\bR\times \bZ)$ is a non-negative function. 

Define
$$ A_1=\{(v,w)\in \bR^2: |H(v)-H(w)|\le 1,  |H(v,w)|^2\le 100 |H(v) H(w)|   \},  $$
and
\begin{equation}\label{eq:def of A_2}
A_2=\{(v,w)\in \bR^2: |H(v)-H(w)|\le 1,  |H(v,w)|^2> 100 |H(v) H(w)|   \}.
\end{equation}
We now show that estimate \eqref{eq:4-linear for hyperbolic RT} will follow once the following two estimates are established:
\begin{equation}\label{eq:main measure est-1}
\l\| \int_{\bR\times \bZ}\ca_{A_1}(v,w){\rm d}w \r\|_{L^\infty_v(\bR\times \bZ)}\lea 1
\end{equation}
and
 \begin{equation}\label{eq:main measure est-2}
\l\| \int_{\bR\times \bZ} \ca_{A_2}(u+v,u-v) {\rm d}u \r\|_{L^\infty_{v}(\bR\times \bZ)}\lea 1 .
\end{equation}

We break
  $\ca_{|H(v)-H(w)|\le 1}$  
into $ \ca_{A_1}(v,w)+\ca_{A_2}(v,w)$. By symmetry, using \eqref{eq:main measure est-1} and the AM–GM inequality, we obtain
\begin{align*}
   & \int_{\bR\times \bZ}\int_{\bR\times \bZ}\int_{\bR\times \bZ}\ca_{A_1}(v,w) f(u+v)f(u-v)f(u+w)f(u-w) {\rm d}u {\rm d}v {\rm d}w \\
\lea & \int_{\bR\times \bZ}\int_{\bR\times \bZ}\int_{\bR\times \bZ}  \ca_{A_1}(v,w) f^2(u+v)f^2(u-v){\rm d}u {\rm d}v {\rm d}w \\
\lea & \|f\|_{L^2}^4.
\end{align*}
On the other hand, by symmetry and applying the AM–GM inequality, we have
\begin{align*}
   & \int_{\bR\times \bZ}\int_{\bR\times \bZ}\int_{\bR\times \bZ}\ca_{A_2}(v,w) f(u+v)f(u-v)f(u+w)f(u-w) {\rm d}u {\rm d}v {\rm d}w \\
\lea & \int_{\bR\times \bZ}\int_{\bR\times \bZ}\int_{\bR\times \bZ}  \ca_{A_2}(v,w) f^2(u+v)f^2(u+w){\rm d}u {\rm d}v {\rm d}w \\
= &  \int_{\bR\times \bZ}\int_{\bR\times \bZ}\int_{\bR\times \bZ}  \ca_{A_2}(v-u,w-u) f^2(v)f^2(w){\rm d}u {\rm d}v {\rm d}w. 
\end{align*}
Thus, by employing a change of coordinates and applying \eqref{eq:main measure est-2}, we complete the proof of \eqref{eq:4-linear for hyperbolic RT}. Note that here we do not distinguish between $\mathbb{R} \times \mathbb{Z}_{1/2}$ and $\mathbb{R} \times \mathbb{Z}$, as they differ only by a constant factor analytically.

It remains to prove \eqref{eq:main measure est-1} and \eqref{eq:main measure est-2}; we will establish these as Proposition \ref{prop:A1} and Proposition \ref{prop:A2-wrong} below.

\end{proof}

\begin{proposition}\label{prop:A1}
  For $v\in \bR\times \bZ$,  define
  $$E(v)=\{w\in \bR^2:  |H(v)-H(w)|\le 1,  |H(v,w)|^2\le 100 |H(v) H(w)|    \}.  $$
  It holds
  $$  \| |E(v)|_{\bR\times \bZ} \|_{L^\infty_v(\bR\times \bZ)} \lea 1. $$
\end{proposition}

\begin{proposition}\label{prop:A2-wrong}
    For $v\in \bR\times \bZ$,  define
\begin{align*}
E(v)= &\{u\in \bR^2:  |H(u+v)-H(u-v)|\le 1,  |H(u+v,u-v)|^2> 100 |H(u+v) H(u-v)|    \}.
\end{align*}
  It holds
  $$  \| |E(v)|_{\bR\times \bZ} \|_{L^\infty_v(\bR\times \bZ)} \lea 1. $$
\end{proposition}

Before proceeding to the proofs of Proposition \ref{prop:A1} and Proposition \ref{prop:A2-wrong}, we provide a brief explanation: under the requirement that the initial data has a frequency cutoff $\le N$, it is not difficult to obtain a version of the estimate in Theorem \ref{thm:hyperbolic on RT} with a loss of $(\log N)^{\frac14}$.  Following the proof of Theorem \ref{thm:hyperbolic on RT} as previously described, and without decomposing into the cases of $A_1$ and $A_2$, the problem reduces to considering the following estimate, which is analogous to \eqref{eq:main measure est-1},
$$ 
\l\| |E(v)|_{\bR\times \bZ} \r\|_{L^\infty_v(\bR\times \bZ)}\lea \log N ,  $$
where
$$ E(v)=\{w\in \bR^2: |H(v)-H(w)|\le 1, |v|, |w|\lea N \}. $$
Indeed, it suffices to fix $|C_0| \lea N^2$ and prove the following measure estimate:
$$ \l| \{w\in \bR^2: |H(w)-C_0|\le 1, |w|\lea N  \} \r|_{\bR\times \bZ}\lea \log N.   $$
Using Lemma \ref{lem:measure control for RZ}, we only need to bound both the Euclidean measure 
\begin{equation}\label{eq:Euclidean measure est-0}
 \l| \{w\in \bR^2: |H(w)-C_0|\le 1, |w|\lea N  \} \r|_{\bR^2}   
\end{equation}
and the length of the longest line segment
\begin{equation}\label{eq:line segment est-0}
\sup_{w_2\in \bZ, |w_2|\lea N} \l| \{w_1\in \bR: |H(w)-C_0|\le 1, |w_1|\lea N  \} \r|_{\bR}.    
\end{equation}
It is not difficult to check \eqref{eq:line segment est-0} is bound by $O(1)$. For \eqref{eq:Euclidean measure est-0}, we change of variables $a=w_1+w_2, b=w_1-w_2$, then
\begin{align*}
\l| \{w\in \bR^2: |H(w)-C_0|\le 1, |w|\lea N  \} \r|_{\bR^2}\sim & |\{(a,b)\in \bR^2: |ab-C_0|\le 1, |a|, |b|\lea N  \}|_{\bR^2} \\
\lea & 1+ \int_1^{CN} \frac{1}{s} {\rm d}s \sim  \log N.
\end{align*}

In addition, if we consider the sharp Strichartz estimate 
$$\|e^{it \partial_{x_1} \partial_{x_2}} P_{\le N} \phi     \|_{L_{t,x}^4([0,1]\times \bR\times \bT)} \lea N^{\frac14} \|\phi\|_{L^2},  $$
we can similarly reduce the problem to proving 
$$\l| \{w\in \bR^2: |\w{H}(w)-C_0|\le 1, |w|\lea N  \} \r|_{\bR\times \bZ}\lea N $$
 by the same method, where $\w{H}(\xi)=\xi_1 \xi_2$ and $|C_0|\lea N$ is a constant. By calculation, it can be seen that in the case where $|C_0|\ll 1$, the length of the longest segment dominates the estimate, satisfying 
 $$\sup_{w_2\in \bZ, |w_2|\lea N} \l| \{w_1\in \bR: |w_1 w_2-C_0|\le 1, |w_1|\lea N  \} \r|_{\bR}\sim N.$$

Now we return to the main topic and proceed to prove Proposition \ref{prop:A1} and Proposition \ref{prop:A2-wrong}. Indeed, the method aligns with what has been outlined above: after applying Lemma \ref{lem:measure control for RZ}, it suffices to compute and establish upper bounds for both the Euclidean measure and the length of the longest line segment. It must be emphasized that we cannot directly prove the length of the longest line segment in $E(v)$ of Proposition \ref{prop:A2-wrong} is $\lea 1$ (although the Euclidean measure bound $\lea 1$ still holds). Therefore, we must impose additional constraints in the defining expression \eqref{eq:def of A_2} for $A_2$. We introduce a refinement of $A_2$ here. Note that
$$ \ca_{|H(v)-H(w)|\le 1}\le \sum_{k\in\bZ} \ca_{|H(v)-k|\le 1}\ca_{|H(w)-k|\le 1}\lea \ca_{|H(v)-H(w)|\le 2}  , $$
so the left hand of \eqref{eq:4-linear for hyperbolic RT} is bounded by
\begin{align*}
& \sum_{k\in \bZ}\int_{\bR\times \bZ} \l(\int_{\bR\times \bZ}\ca_{|H(v)-k|\le 1}f(u+v)f(u-v){\rm d}v\r)^2 {\rm d}u\\
\lea & \sum_{j\in \{-1,1\}^3}\sum_{k\in \bZ}\int_{\bR\times \bZ} \l(\int_{\bR\times \bZ}\ca_{|H(v)-k|\le 1} \ca_{J_j}(v)   f(u+v)f(u-v){\rm d}v\r)^2 {\rm d}u \\
\lea & \sum_{j\in \{-1,1\}^3}\int_{\bR \times \bZ}\int_{\bR \times \bZ}\int_{\bR \times \bZ} \ca_{|H(v)-H(w)|\le 2} \ca_{J_j}(v) \ca_{J_j}(w) f(u+v)f(u-v)f(u+w)f(u-w) {\rm d}u {\rm d}v {\rm d}w,
\end{align*}
where $J_{j}$ is defined by $\{v\in \bR^2: j_1 v_1 \ge 0, j_2 v_2\ge 0, j_3 H(v)\ge 0   \}$, $j\in \{-1,1\}^3$. Thus, following the previous approach in the proof of Theorem \ref{thm:hyperbolic on RT}, we now restrict our consideration to the case where the definition of $A_2$ in \eqref{eq:main measure est-2} is refined as 
$$  A_2=\{(v,w)\in \bR^2: |H(v)-H(w)|\le 2,  |H(v,w)|^2> 100 |H(v) H(w)|, v\in J_j, w\in J_j   \}  $$
for a fixed $j\in \{-1,1\}^3$.
The scenario for $j=(1,1,1)$ will be addressed in Proposition \ref{prop:A2}, while other cases follow a similar approach.

\begin{proposition}\label{prop:A2}
    For $v\in \bR\times \bZ$,  define
\begin{align*}
E(v)= &\{u\in \bR^2:  |H(u+v)-H(u-v)|\le 2,  |H(u+v,u-v)|^2> 100 |H(u+v) H(u-v)|    \}  \\
& \bigcap \{u\in \bR^2: u_1+v_1\ge  u_2+v_2\ge 0, u_1-v_1\ge  u_2-v_2\ge 0     \}.
\end{align*}
  It holds
  $$  \| |E(v)|_{\bR\times \bZ} \|_{L^\infty_v(\bR\times \bZ)} \lea 1. $$
\end{proposition}

\begin{proof}[Proof of Proposition \ref{prop:A1}]
Fix $v\in \bR\times \bZ$ with $H(v)\ne 0$. Applying Lemma \ref{lem:measure control for RZ}, it suffices to prove the Euclidean measure estimate
\begin{equation} \label{eq:Euclidean measure est-1}
\l| \{w\in \bR^2:  |H(v)-H(w)|\le 1,  |H(v,w)|^2\le 100 |H(v) H(w)|    \} \r|_{\bR^2}\lea 1
\end{equation}
and the estimate for the length of the longest line segment
\begin{equation} \label{eq:line segment est-1}
\sup_{w_2\in \bZ}\l| \{w_1\in \bR:  |H(v)-H(w)|\le 1,  |H(v,w)|^2\le 100 |H(v) H(w)|    \} \r|_{\bR}\lea 1.
\end{equation}

\underline{ Proof of \eqref{eq:Euclidean measure est-1}. } 
We write 
$$|H(v,w)|^2\le 100 |H(v) H(w)| $$
as
$$ |(w_1+w_2)(v_1-v_2)+(w_1-w_2)(v_1+v_2) |\lea \sqrt{|(w_1+w_2)(v_1-v_2)(w_1-w_2)(v_1+v_2)|}  ,$$
thus demonstrating that 
$$ |(w_1+w_2)(v_1-v_2)|\sim |(w_1-w_2)(v_1+v_2)|   $$
holds.

Define $c=v_1+v_2$ and $ d=v_1-v_2$. Since $H(v)\ne 0$, we have  $cd\ne 0$. After performing the change of variables $\alpha=\frac{w_1+w_2}{c}, \beta=\frac{w_1-w_2}{d}$, \eqref{eq:Euclidean measure est-1} reduces to
\begin{equation}\label{eq:Euclidean measure est-1-1}
 \l|\{(\alpha,\beta)\in \bR^2: |\alpha\beta-1|\le |cd|^{-1}, |\alpha|\sim |\beta|   \}          \r|_{\bR^2} \lea |cd|^{-1}. 
\end{equation}

{\bf Case 1)}: $|cd|\lea 1$.  Then we can deduce that $ 
|\alpha|, |\beta|\lea |cd|^{-\frac12}$ by $|\alpha|\sim |\beta|$ and 
$$|\alpha \beta|\le 1+|\alpha\beta-1|\le |cd|^{-1}.$$ Thus \eqref{eq:Euclidean measure est-1-1} holds.

{\bf Case 2)}: $|cd|\gg 1$. The diagram of the region corresponding to Case 2 can be found in Figure 2.
We can deduce that
$|\alpha|, |\beta|\sim 1$ by $|\alpha|\sim |\beta|$ and 
$$|\alpha\beta-1|\le |cd|^{-1} \ll 1.$$
So
\begin{align*}
&\l|\{(\alpha,\beta)\in \bR^2: |\alpha\beta-1|\le |cd|^{-1}, |\alpha|\sim |\beta|   \}          \r|_{\bR^2}  \\  
\lea & \int_{|\alpha|\sim 1}  |cd|^{-1} |\alpha|^{-1} {\rm d}\alpha 
\sim |cd|^{-1}.
\end{align*}
We have proved \eqref{eq:Euclidean measure est-1-1}.

\underline{ Proof of \eqref{eq:line segment est-1}. } It holds
\begin{align*}
&\sup_{w_2\in \bZ}\l| \{w_1\in \bR:  |H(v)-H(w)|\le 1,  |H(v,w)|^2\le 100 |H(v) H(w)|    \} \r|_{\bR} \\
\lea & \sup_{C_0\in \bR}\l| \{w_1\in \bR:  |w_1^2-C_0|\le 1   \} \r|_{\bR} \lea 1.
\end{align*}
\end{proof}

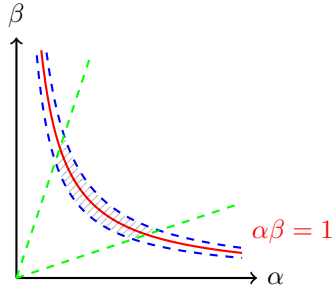
\begin{figure}[ht]
\centering
\begin{tikzpicture}[scale=1]
  
  \draw[->, thick] (0,0) -- (3.2,0) node[right] {$\alpha$};
  \draw[->, thick] (0,0) -- (0,3.2) node[above] {$\beta$};
  
  \draw[red, thick, domain=0.33:3, samples=100] plot (\x, {1/\x}) node[above right] {$\alpha\beta=1$};
  
  \draw[blue, dashed, thick, domain=0.27:3, samples=100] plot (\x, {0.8/\x});
  \draw[blue, dashed, thick, domain=0.4:3, samples=100] plot (\x, {1.2/\x});
  
  \draw[green, dashed, thick] (0,0) -- (3,1) node[right] {};
  \draw[green, dashed, thick] (0,0) -- (1,3) node[above right] {};
  
  \fill[pattern=north east lines, pattern color=black, opacity=0.5] 
    plot[domain=0.516:1.549, samples=100] (\x, {0.8/\x})
    -- plot[domain=1.549:1.897, samples=100] (\x, {\x/3})
    -- plot[domain=1.897:0.632, samples=100] (\x, {1.2/\x})
    -- plot[domain=0.632:0.516, samples=100] (\x, {3*\x})
    -- cycle;
  
\end{tikzpicture}
\caption{Intersection of regions $|\alpha\beta-1|\le |cd|^{-1}$ and $|\alpha|\sim |\beta|$.}
\end{figure}

\begin{proof}[Proof of Proposition \ref{prop:A2}]
 Fix $v\in \bR\times \bZ$ with $H(v)\ne 0$. Applying Lemma \ref{lem:measure control for RZ}, it suffices to prove the Euclidean measure estimate
\begin{equation} \label{eq:Euclidean measure est-2}
\l| \{u\in \bR^2:  |H(u+v)-H(u-v)|\le 2,  |H(u+v,u-v)|^2> 100 |H(u+v) H(u-v)|    \} \r|_{\bR^2}\lea 1
\end{equation}
and the estimate for the length of the longest line segment
\begin{align}\label{eq:line segment est-2}
\sup_{u_2\in \bZ}\Big|  &\{u_1\in \bR:  |H(u+v)-H(u-v)|\le 2,  |H(u+v,u-v)|^2> 100 |H(u+v) H(u-v)|     \}  \\
& \bigcap \{u_1\in \bR: u_1+v_1\ge  u_2+v_2\ge 0, u_1-v_1\ge  u_2-v_2\ge 0     \}
\Big|_{\bR}\lea 1.
\end{align}

\underline{ Proof of \eqref{eq:Euclidean measure est-2}. } 
Similar to the discussion in the proof of \eqref{eq:Euclidean measure est-1}, it follows from 
$$ |H(u+v,u-v)|^2> 100 |H(u+v) H(u-v)|    $$
that we may assume, without loss of generality, that 
\begin{equation}\label{eq:equivalent condition}
|(u_1+u_2+v_1+v_2)(u_1-u_2-v_1+v_2)| \gg |(u_1+u_2-v_1-v_2)(u_1-u_2+v_1-v_2)|     
\end{equation}
holds.

Define $c=v_1+v_2$ and $ d=v_1-v_2$. Since $H(v)\ne 0$, we have  $cd\ne 0$. After performing the change of variables $\alpha=\frac{u_1+u_2}{c}, \beta=\frac{u_1-u_2}{d}$, \eqref{eq:Euclidean measure est-2} reduces to
\begin{equation}\label{eq:Euclidean measure est-2-1}
 \l|\{(\alpha,\beta)\in \bR^2: |\alpha+\beta|\lea |cd|^{-1}, |(\alpha+1)(\beta-1)|\gg |(\alpha-1)(\beta+1)|   \}          \r|_{\bR^2} \lea |cd|^{-1}. 
\end{equation}
When $|(\alpha+1)(\beta-1)|\gg |(\alpha-1)(\beta+1)|$ holds, $|\alpha|\gg 1$ and $|\beta|\gg 1$ cannot both be true. Furthermore, it can be shown that when either $\alpha$ or $\beta$ is fixed, the length of the interval in which the other variable takes values is $O(|cd|^{-1})$. Combining the two facts above, we establish \eqref{eq:Euclidean measure est-2-1}.

\underline{ Proof of \eqref{eq:line segment est-2}. } Fix $u_2\in \bZ$. The condition 
$$ u_1+v_1\ge  u_2+v_2\ge 0, u_1-v_1\ge  u_2-v_2\ge 0  $$
implies 
\begin{equation}\label{eq:line segment est-2-1}
u_1-  u_2\ge |v_1-v_2|, \quad u_2\ge |v_2|.
\end{equation}
As in the analysis of \eqref{eq:Euclidean measure est-2-1}, we deduce from condition
$$ |H(u+v,u-v)|^2> 100 |H(u+v) H(u-v)|    $$
 that either $|\alpha|\lea 1$ or $|\beta|\lea 1$ must hold, where $\alpha=\frac{u_1+u_2}{v_1+v_2}$ and $\beta=\frac{u_1-u_2}{v_1-v_2}$. We may assume without loss of generality that 
 \begin{equation}\label{eq:line segment est-2-2}
 |u_1+u_2|\lea |v_1+v_2|
 \end{equation}
 holds. A calculation shows that condition 
 $$|H(u+v)-H(u-v)|\le 2 $$
 is equivalent to 
 \begin{equation}\label{eq:line segment est-2-3}
  |u_1v_1-u_2v_2|\lea 1.
 \end{equation}
 
 {\bf Case 1)}: $|v_1|\gtrsim 1$. Hence, estimate \eqref{eq:line segment est-2} follows from \eqref{eq:line segment est-2-3}.
 
 {\bf Case 2)}: $|v_1|\ll 1$ and $|v_2|\lea 1$. We can deduce that \eqref{eq:line segment est-2} holds by \eqref{eq:line segment est-2-2}.
 
 {\bf Case 3)}: $|v_1|\ll 1$ and $|v_2|\gg 1$. From \eqref{eq:line segment est-2-1}, we see that $|u_2|\ge |v_2|\gg 1$ holds. It then follows easily from \eqref{eq:line segment est-2-2} that 
 $$|u_1|\lea |u_2|+|v_1+v_2|\lea |u_2|.$$
 Accordingly, in the case analysis we can ensure that $$\l|\frac{u_1v_1}{u_2 v_2}\r| \ll 1$$ holds by choosing the implicit constant appropriately. Thus it can be concluded that the set involved in \eqref{eq:line segment est-2} is empty, since the following contradiction arises: 
 $$ 1 \gtrsim |u_1v_1-u_2v_2|=|u_2v_2| \l|1- \frac{u_1v_1}{u_2 v_2}  \r| \gtrsim |u_2v_2| \gg 1 . $$
\end{proof}

\begin{remark}
In fact, following the proof of Theorem \ref{thm:hyperbolic on RT}, we can also prove the following sharp estimate, improving Theorem 1.8 of \cite{BSTW25}:
$$ \|e^{it \partial_{x_1} \partial_{x_2}}  \phi     \|_{L_{t,x}^4([0,1]\times \bR\times \bT)} \lea  \|\phi\|_{L^2},   $$
where the initial data $\phi(x_1, x_2)$ is mean-zero in $x_2$.  This can be achieved by adjusting the choices of 
$A_1$ and $A_2$ in the proof of Theorem \ref{thm:hyperbolic on RT} under a suitable change of variables. The Euclidean measure estimates are equivalent to \eqref{eq:Euclidean measure est-1} and \eqref{eq:Euclidean measure est-2} under this variable transformation, while the length of the longest line segment in this setting can be readily estimated.

\end{remark}

\section{Well-posedness in $L^2$-critical space: Proof of Theorem \ref{mainthm: 2}}\label{4}
In this section, we establish the well-posedness theory for \eqref{eq: main} by giving the proof of Theorem \ref{mainthm: 2}. The well-posedness theory follows by standard arguments once the sharp $L^{4}$ Strichartz estimate is established; therefore, we only sketch the proof. See Takaoka-Tzvetkov \cite{takaoka20012d} for more details.

First, using the $T$-$T^{\ast}$ argument and Christ-Kiselev lemma, it is standard to obtain the inhomogeneous version of the Strichartz estimate as follows.

\begin{equation}
   \big\| \int_0^t H(t-s) f(s)ds \big\|_{L_{t,x}^{4}} \lesssim \|f\|_{L_{t,x}^{\frac{4}{3}}}.
\end{equation}

Next, define a sequence $u_n$ as follows
\begin{equation}
    u_0=H(t)\phi,
\end{equation}
\begin{equation}
    u_{n+1}=H(t)u_n+i\int_0^t H(t-s)\big(|u_n(s)|^2u_n(s)\big)ds.
\end{equation}
Take $I=[-1,1]$. Then we obtain
\begin{equation}
   \|u_{n+1}\|_{L_{t,x}^4(I)}\leq C(I)\|\phi\|_{L^2}+C(I) \|u_{n}\|^3_{L_{t,x}^{4}(I)}
\end{equation}
and
\begin{equation}
   \|u_{n+1}-u_n\|_{L_{t,x}^4(I)}\leq C(I) \|u_{n+1}-u_n\|_{L_{t,x}^4(I)}\times (\|u_{n+1}\|^2_{L_{t,x}^{4}(I)}+\|u_{n}\|^2_{L_{t,x}^{4}(I)}). 
\end{equation}
Hence taking $\|\phi\|_{L^2}$ small, we obtain that $u_n$ is a Cauchy sequence in a small ball close to the origin of $L_t^{\infty}L^2_{x}\cap L_{t,x}^4(I)$. Therefore $u_n$ converges to a local solution in the time interval $[-1,1]$. Due to the mass conservation law, an iteration of the last process giving local solutions extends the solutions globally in time.

\section{Sharp Bilinear Strichartz estimate: Proof of Theorem \ref{mainthm: 3}}\label{5}
In this section, building on Lemma \ref{lem:measure control for RZ}, we present a concise proof of the sharp bilinear estimate stated in Theorem \ref{mainthm: 3}. The proof parallels the argument for the elliptic case, with Lemma \ref{lem:measure control for RZ} replacing the number-theoretic counting argument.

\begin{proof}[Proof of Theorem \ref{mainthm: 3}]
By the standard argument, (see \cite{DPST} or \cite{takaoka20012d}), we only need to prove the measure estimate
$$  \sup_{|a|\sim N_1, |b|\sim N_2} |E_{a,b}|_{\bR\times \bZ_{1/\lambda}} \lea \frac{1}{\lambda}+\frac{N_2}{N_1},        $$
where
$$E_{a,b}= \l\{\eta \in \bR^2:  \l|H(\eta)+H(a+b-\eta)-H(a)-H(b)  \r|\lea 1, |\eta|\sim N_2   \r\}. $$
Fix $a,b$ with $|a|\sim N_1, |b|\sim N_2$.  Using Lemma \ref{lem:measure control for RZ}, it suffices to show
\begin{equation}\label{eq:Eab}
|E_{a,b}|_{\bR^2}\lea \frac{N_2}{N_1}.
\end{equation}
Assume $|a_1|\sim N_1$ without loss of generality. If $\eta=(\eta_1, \eta_2), \eta'= (\eta_1', \eta_2')\in E_{a,b} $ with $\eta_2=\eta_2'$, then
$$     \l|H(\eta)+H(a+b-\eta)-H(\eta')-H(a+b-\eta')  \r|\lea 1,                $$
which implies
$$ |(\eta_1-\eta_1')(a_1+b_1-\eta_1-\eta_1')    |\lea 1 .$$
Note that $N_1 \gg N_2$, there holds $ |a_1+b_1-\eta_1-\eta_1'|\sim N_1  $, and then
$ |\eta_1-\eta_1'| \lea \frac{1}{N_1}           $. 

The above calculation shows that when $(\eta_1, \eta_2)\in E_{a,b}$ and $\eta_2$ is fixed, the length of the existence interval for $\eta_1$ does not exceed $\frac{1}{N_1}$, which implies that
$$ |E_{a,b}|_{\bR^2}\lea \int_{|\eta_2|\lea N_2} \frac{1}{N_1} {\rm d}\eta_2 \lea \frac{N_2}{N_1}.    $$
Thus \eqref{eq:Eab} holds. The proof of Theorem \ref{mainthm: 3} is now complete.
 
\end{proof}

\section{Summary and Open problems}\label{6}

Finally, we make a summary and comment on some related open problems.

The current paper fills a gap in the theory of sharp Strichartz estimates for Schr\"odinger equations, both elliptic and hyperbolic for Euclidean spaces, tori, and waveguides, as summarized below.
\begin{table}[htbp]\label{F1}
\centering
\renewcommand{\arraystretch}{1.3}
\begin{tabular}{|c|c|c|c|}
\hline
\textbf{Geometry} & \textbf{Operator} & \textbf{Sharp $L^4$ Strichartz} & \textbf{Reference} \\
\hline
$\mathbb{R}^2$ & $-\Delta$ (elliptic) & no loss & Strichartz \cite{Strichartz1977}  \\
\hline
$\mathbb{R}^2$ & $\square=\partial_x^2-\partial_y^2$ (hyperbolic) & no loss & Strichartz \cite{Strichartz1977} \\
\hline
$\mathbb{T}^2$ & $-\Delta$ (elliptic) & $\log^{\frac{1}{4}}(N)$-loss & Herr-Kwak \cite{Herr}\\
\hline
$\mathbb{T}^2$ & $\square=\partial_x^2-\partial_y^2$ (hyperbolic) & $N^{\frac{1}{4}}$-loss & Wang \cite{wang2013periodic} \\
\hline
$\mathbb{R}\times\mathbb{T}$ & $-\Delta$ (elliptic) & no loss & Takaoka--Tzvetkov \cite{takaoka20012d} \\
\hline
$\mathbb{R}\times\mathbb{T}$ & $\square=\partial_x^2-\partial_y^2$ (hyperbolic) & no loss & \textbf{This work} \\
\hline
\end{tabular}
\caption{Sharp Strichartz estimates for Schr\"odinger equations in canonical two-dimensional geometries.}
\end{table}

As mentioned in the introduction, a central challenge may be global well-posedness for the cubic hyperbolic NLS on $\mathbb{R}\times\mathbb{T}$ with large $L^2$ data. Indeed, global well-posedness for large $H^{s}$ data is also non-trivial, due to the lack of a coercive conservation law, stemming from the hyperbolic signature.

 Extending the present results to higher-dimensional waveguides such as $\mathbb{R}^m\times\mathbb{T}^n$ also poses intriguing challenges: determining the sharp Strichartz exponent in these geometries would shed light on how the interplay between Euclidean and compact components affects dispersion. The methods in the current article also cannot be directly applied in the fractional case.

Another natural direction is to investigate multilinear refinements and potential applications to the Davey–Stewartson system (and other models), where hyperbolic dispersion also plays a prominent role.

\subsection*{Data Availability Statement.} Data sharing not applicable to this article as no datasets were generated or analysed during the current study.

 \subsection*{Conflict of Interest} The authors declare no competing interests.

\bibliographystyle{abbrv}

\bibliography{ref}

\end{document}